\documentclass[a4paper,11pt]{article}
\usepackage{braket,hyperref}
\usepackage{amssymb,mathtools,amsthm,amsmath}
\usepackage{tikz}

\usepackage{xargs}

\usepackage{enumitem}
\usepackage[colorinlistoftodos,prependcaption,textsize=small]{todonotes}
\usepackage{verbatim}

\theoremstyle{plain}
\newtheorem{theorem}{\bf Theorem}
\newtheorem{claim}[theorem]{\bf Claim}

\newtheorem{proposition}[theorem]{\bf Proposition}
\newtheorem{corollary}[theorem]{\bf Corollary}
\newtheorem{lemma}[theorem]{\bf Lemma}
\theoremstyle{definition}
\newtheorem{definition}[theorem]{\bf Definition}

\numberwithin{theorem}{section}
\numberwithin{equation}{section}

\DeclareMathOperator{\lk}{lk}

\DeclareMathOperator{\mes}{mes}

\newcommand{\cupdot}{\mathbin{\mathaccent\cdot\cup}}

\DeclareMathOperator{\intcomplex}{Int}
\DeclareMathOperator{\hcov}{Cov}

\begin{document}
	
		\title{Collapsibility of simplicial complexes of hypergraphs}
		\author{Alan Lew\footnote{Department of Mathematics, Technion, Haifa 32000, Israel.		
				e-mail: alan@campus.technion.ac.il .  Supported by ISF grant no. 326/16.}}
		
		\date{}

\maketitle

\begin{abstract}

Let $\mathcal{H}$ be an $r$-uniform hypergraph.
We show that the simplicial complex whose simplices are the hypergraphs $\mathcal{F}\subset\mathcal{H}$ with covering number at most $p$ is $\left(\binom{r+p}{r}-1\right)$-collapsible. Similarly, the simplicial complex whose simplices are the pairwise intersecting hypergraphs $\mathcal{F}\subset\mathcal{H}$ is $\frac{1}{2}\binom{2r}{r}$-collapsible.

\end{abstract}

\section{Introduction}

Let $X$ be a finite simplicial complex. Let $\eta$ be a simplex of $X$ such that  $|\eta|\leq d$ and $\eta$ is contained in a unique maximal face $\tau\in X$. We say that the complex
\[
    X'= X\setminus \{ \sigma\in X:\, \eta\subset\sigma\subset\tau\}
\]
is obtained from $X$ by an \emph{elementary $d$-collapse}, and we write
$
    X\xrightarrow{\eta} X'.
$

The complex $X$ is called \emph{$d$-collapsible} if there exists a sequence of elementary $d$-collapses 
\[
    X=X_1\xrightarrow{\eta_1} X_2 \xrightarrow{\eta_2} \cdots \xrightarrow{\eta_{k-1}} X_k=\emptyset
\]
from $X$ to the void complex $\emptyset$. The \emph{collapsibility} of $X$ is the minimal $d$ such that $X$ is $d$-collapsible.

A simple consequence of $d$-collapsibility is the following:

\begin{proposition}[Wegner {\cite[Lemma 1]{wegner1975d}}]
	\label{claim:hom}
	If $X$ is $d$-collapsible then it is homotopy equivalent to a simplicial complex of dimension smaller than $d$. 
\end{proposition}

Let $\mathcal{H}$ be a finite hypergraph. We identify $\mathcal{H}$ with its edge set. The \emph{rank} of $\mathcal{H}$ is the maximal size of an edge of $\mathcal{H}$.

A set $C$ is a \emph{cover} of $\mathcal{H}$ if $A\cap C\neq \emptyset$ for all $A\in\mathcal{H}$. The \emph{covering number} of $\mathcal{H}$, denoted by $\tau(\mathcal{H})$, is the minimal size of a cover of $\mathcal{H}$.

For $p\in \mathbb{N}$, let
\[
	\hcov_{\mathcal{H},p}=\{ \mathcal{F}\subset\mathcal{H} : \, \tau(\mathcal{F})\leq p\}.
\]
That is, $\hcov_{\mathcal{H},p}$ is a simplicial complex whose vertices are the edges of $\mathcal{H}$ and whose simplices are the hypergraphs $\mathcal{F}\subset\mathcal{H}$ that can be covered by a set of size at most $p$.
Some topological properties of the complex $\hcov_{\binom{[n]}{r},p}$ were studied by Jonsson in \cite{jonsson2008simplicial}.

The hypergraph $\mathcal{H}$ is called \emph{pairwise intersecting} if $A\cap B\neq \emptyset$ for all $A,B\in \mathcal{H}$.
Let
\[
\intcomplex_{\mathcal{H}}=\{\mathcal{F}\subset\mathcal{H} : \, A\cap B\neq \emptyset \text{ for all } A,B\in\mathcal{F}\}.
\]
So, $\intcomplex_{\mathcal{H}}$ is a simplicial complex whose vertices are the edges of $\mathcal{H}$ and whose simplices are the hypergraphs $\mathcal{F}\subset\mathcal{H}$ that are pairwise intersecting.

Our main results are the following:

\begin{theorem}
	\label{thm:cov}
	Let $\mathcal{H}$ be a hypergraph of rank $r$. Then $\hcov_{\mathcal{H},p}$ is $\left(\binom{r+p}{r}-1\right)$-collapsible.
\end{theorem}

\begin{theorem}
	\label{thm:int}
	Let $\mathcal{H}$ be a hypergraph of rank $r$. Then  $\intcomplex_{\mathcal{H}}$ is $\frac{1}{2}\binom{2r}{r}$-collapsible.
\end{theorem}

The following examples show that these bounds are sharp:
\begin{itemize}
\item Let $\mathcal{H}=\binom{[r+p]}{r}$ be the complete $r$-uniform hypergraph on $r+p$ vertices. The covering number of $\mathcal{H}$ is $p+1$, but for any $A\in\mathcal{H}$ the hypergraph $\mathcal{H}\setminus\{A\}$ can be covered by a set of size $p$, namely by $[r+p]\setminus A$. Therefore the complex $\hcov_{\binom{[r+p]}{r},p}$ is the boundary of the $\left(\binom{r+p}{r}-1\right)$-dimensional simplex, so it is homeomorphic to a $\left(\binom{r+p}{r}-2\right)$-dimensional sphere. 
 Hence, by Proposition \ref{claim:hom}, $\hcov_{\binom{[r+p]}{r},p}$ is not $\left(\binom{r+p}{r}-2\right)$-collapsible.
 
\item Let $\mathcal{H}=\binom{[2r]}{r}$ be the complete $r$-uniform hypergraph on $2r$ vertices. Any $A\in \mathcal{H}$ intersects all the edges of $\mathcal{H}$ except the edge $[2r]\setminus A$. Therefore the complex $\intcomplex_{\binom{[2r]}{r}}$ is the boundary of the $\frac{1}{2}\binom{2r}{r}$-dimensional cross-polytope, so it is homeomorphic to a $\left( \frac{1}{2}\binom{2r}{r}-1\right)$-dimensional sphere. Hence, by Proposition \ref{claim:hom}, $\intcomplex_{\binom{[2r]}{r}}$ is not $\left(\frac{1}{2}\binom{2r}{r}-1\right)$-collapsible.
\end{itemize}

A related problem was studied by Aharoni, Holzman and Jiang in \cite{aharoni2018fractional}, where they show that for any $r$-uniform hypergraph $\mathcal{H}$ and $p\in\mathbb{Q}$, the complex of hypergraphs $\mathcal{F}\subset\mathcal{H}$ with fractional matching number (or equivalently, fractional covering number) smaller than $p$ is $(\lceil rp \rceil -1)$-collapsible.

Our proofs rely on two main ingredients. The first one is the following theorem:

\begin{theorem}\label{thm:mes_simpler}
Let $X$ be a simplicial complex on vertex set $V$. Let $S(X)$ be the collection of all sets $\{v_1,\ldots,v_k\}\subset V$ satisfying the following condition:
\begin{itemize}
    \item[]
There exist maximal faces $\sigma_1,\sigma_2,\ldots,\sigma_{k+1}$ of $X$ such that:
    \begin{itemize}
        \item[$\bullet$] $v_i \notin \sigma_i$ for all $i\in[k]$,
        \item[$\bullet$] $v_i \in \sigma_j$ for all $1\leq i<j\leq k+1$.
    \end{itemize}
\end{itemize}

Let $d'(X)$ be the maximum size of a set in $S(X)$. Then $X$ is $d'(X)$-collapsible.
\end{theorem}

Theorem \ref{thm:mes_simpler} is a special case of a more general result, due essentially to  Matou{\v{s}}ek and Tancer  (who stated it in the special case where the complex is the nerve of a family of finite sets, and used it to prove the case $p=1$ of Theorem \ref{thm:cov}; see \cite{matousek2008dimension}).

The second ingredient is the following combinatorial lemma, proved independently by Frankl and Kalai.

\begin{lemma}[Frankl \cite{frankl1982anextremal}, Kalai \cite{kalai1984intersection}]
	\label{lem:frankl}
	Let $\{A_1,\ldots, A_k\}$ and $\{B_1,\ldots, B_k\}$ be families of sets such that:
	\begin{itemize}
	\item 	$|A_i|\leq r$, $|B_i|\leq p$ for all $i\in[k]$,
	\item $A_i\cap B_i=\emptyset$ for all $i\in [k]$, \item $A_i\cap B_j\neq \emptyset$ for all $1\leq i<j\leq k$.
	\end{itemize}
	Then
	\[
	k\leq \binom{r+p}{r}.
	\]
\end{lemma}

The paper is organized as follows. In Section \ref{sec:2} we present Matou{\v{s}}ek and Tancer's bound on the collapsibility of a simplicial complex, and we prove Theorem \ref{thm:mes_simpler}.
In Section \ref{sec:graphs} we present some results on the collapsibility of independence complexes of graphs.
 In Section \ref{sec:3} we prove our main results on the collapsibility of complexes of hypergraphs. Section \ref{sec:4} contains some generalizations of Theorems \ref{thm:cov} and \ref{thm:int}, which are obtained by applying different known variants of Lemma \ref{lem:frankl}.

\section{A bound on the collapsibility of a complex}
\label{sec:2}

Let $X$ be a (non-void) simplicial complex on vertex set $V$. Fix a linear order $<$ on $V$. Let $\mathcal{A}=(\sigma_1,\ldots,\sigma_m)$ be a sequence of faces of $X$ such that, for any $\sigma\in X$, $\sigma\subset\sigma_i$ for some $i\in[m]$. For example, we may take $\sigma_1,\ldots,\sigma_m$ to be the set of maximal faces of $X$ (ordered in any way).

For a simplex $\sigma\in X$, let $m_{X,\mathcal{A},<}(\sigma)=\min\{i\in[m]: \, \sigma\subset \sigma_i\}$.
Let $i\in[m]$ and $\sigma\in X$ such that $m_{X,\mathcal{A},<}(\sigma)=i$. We define the \emph{minimal exclusion sequence}
\[
\mes_{X,\mathcal{A},<}(\sigma)=(v_1,\ldots,v_{i-1})
\]
as follows:
If $i=1$ then $\mes_{X,\mathcal{A},<}(\sigma)$ is the empty sequence. If $i>1$ we define the sequence recursively as follows:

Since $i>1$, we must have $\sigma\not\subset \sigma_1$; hence, there is some $v\in \sigma$ such that $v\notin \sigma_1$. Let $v_1$ be the minimal such vertex (with respect to the order $<$).

Let $1<j<i$ and assume that we already defined $v_1,\ldots,v_{j-1}$. Since $i>j$, we must have $\sigma\not\subset\sigma_j$; hence, there exists some $v\in \sigma$ such that $v\notin \sigma_j$.
 \begin{itemize}
     \item  If there is a vertex $v_k\in \{v_1,\ldots, v_{j-1}\}$ such that $v_k\notin \sigma_j$, let $v_j$ be such a vertex of minimal index $k$. In this case we call $v_j$ \emph{old at $j$}.
     \item If $v_k\in\sigma_j$ for all $k<j$, let $v_j$ be the minimal vertex $v\in\sigma$ (with respect to the order $<$) such that $v\notin \sigma_j$. In this case we call $v_j$ \emph{new at $j$}.
     \end{itemize}

Let $M_{X,\mathcal{A},<}(\sigma)\subset\sigma$ be the simplex consisting of all the vertices appearing in the sequence $\mes_{X,\mathcal{A},<}(\sigma)$.
Let
\[
    d(X,\mathcal{A},<)=\max\{ |M_{X,\mathcal{A},<}(\sigma)|: \, \sigma\in X\}.
\]

The following result was stated and proved in \cite[Prop. 1.3]{matousek2008dimension} in the special case where $X$ is the nerve of a finite family of sets (in our notation, $X=\hcov_{\mathcal{H},1}$ for some hypergraph $\mathcal{H}$). 

\begin{theorem}	\label{prop:collapsible_by_mes}
The simplicial complex $X$ is $d(X,\mathcal{A},<)$-collapsible.
\end{theorem}

The proof given in \cite{matousek2008dimension} can be easily modified to hold in this more general setting. Here we present a different proof.

Let $X$ be a simplicial complex on vertex set $V$, and let $v\in V$. Let
\[
    X\setminus v= \{ \sigma\in X:\, v\notin \sigma\}
\]
and
\[
    \lk(X,v)=\{\sigma\in X:\, v\notin\sigma,\, \sigma\cup\{v\}\in X\}.
\]

We will need the following lemma, proved by Tancer in \cite{tancer2011strong}:

\begin{lemma}[Tancer {\cite[Prop. 1.2]{tancer2011strong}}]
\label{lemma:tancer}
If $X\setminus v$ is $d$-collapsible and $\lk(X,v)$ is $(d-1)$-collapsible, then $X$ is $d$-collapsible.
\end{lemma}

\begin{proof}[Proof of Theorem \ref{prop:collapsible_by_mes}]

First, we deal with the case where $X$ is a complete complex (i.e. a simplex). Then $X$ is $0$-collapsible; therefore, the claim holds.

For a general complex $X$, we argue by induction on the number of vertices of $X$. If $|V|=0$, then $X=\{\emptyset\}$. In particular, it is a complete complex; hence, the claim holds.

Let $|V|>0$, and assume that the claim holds for any complex with less than $|V|$ vertices. If $\sigma_1=V$, then $X$ is the complete complex on vertex set $V$, and the claim holds. Otherwise, let $v$ be the minimal vertex (with respect to $<$) in $V\setminus\sigma_1$.

In order to apply Lemma \ref{lemma:tancer}, we will need the following two claims:

\begin{claim}\label{claim:deletion}
 The complex $X\setminus v$ is $d(X, \mathcal{A},<)$-collapsible.
\end{claim}
\begin{proof}
For every $i\in[m]$, let $\sigma_i'=\sigma_i\setminus \{v\}$, and let $\mathcal{A}'=(\sigma_1',\ldots,\sigma_m')$.
Let $\sigma\in X\setminus v$. Since $v\notin \sigma$, then, for any $i\in[m]$, $\sigma\subset \sigma_i$ if and only if $\sigma\subset\sigma_i'$. 
Hence, every simplex  $\sigma\in X\setminus v$ is contained in $\sigma_i'$ for some $i\in[m]$ (since, by the definition of $\mathcal{A}$, $\sigma\subset \sigma_i$ for some $i\in[m]$). So, by the induction hypothesis, $X\setminus v$ is $d(X\setminus v,\mathcal{A}',<)$-collapsible.

Let $\sigma\in X\setminus v$.  We will show that $\mes_{X,\mathcal{A},<}(\sigma)=\mes_{X\setminus v,\mathcal{A}',<}(\sigma)$. Since for any $i\in[m]$, $\sigma\subset \sigma_i$ if and only if $\sigma\subset \sigma_i'$, then the two sequences are of the same length. Let 
\[\mes_{X,\mathcal{A},<}(\sigma)=(v_1,\ldots,v_{k})\] 
and 
\[\mes_{X\setminus v,\mathcal{A}',<}(\sigma)=(v'_1,\ldots,v'_{k}).\]
We will show that $v_i=v'_i$ for all $i\in[k]$. We argue by induction on $i$. Let $i\in [k]$, and assume that $v_{j}=v'_{j}$ for all $j<i$. 
Since $v\notin \sigma$, then $\sigma\setminus\sigma_i=\sigma\setminus\sigma_i'$. Therefore, for any $j<i$,  $v_{j}\in \sigma\setminus \sigma_i$ if and only if $v_{j}'=v_{j}\in\sigma\setminus \sigma_i'$. Hence, $v_i$ is old at $i$ if and only if $v_i'$ is old at $i$, and if $v_i$ and $v_i'$ are both old at $i$, then $v_i=v_i'$. Otherwise, 
both $v_i$ and $v_i'$ are new at $i$. Then, $v_i$ is the minimal vertex in $\sigma\setminus \sigma_i$, and $v_i'$ is the minimal vertex in $\sigma\setminus\sigma_i'=\sigma\setminus\sigma_i$. Thus, $v_i=v_i'$.

Therefore, $|M_{X\setminus v, \mathcal{A}',<}(\sigma)|=|M_{X, \mathcal{A},<}(\sigma)|$ for any $\sigma\in X\setminus v$; hence,
\[
d(X\setminus v,\mathcal{A}',<)\leq d(X,\mathcal{A},<).\]
So, $X\setminus v$ is $d(X,\mathcal{A},<)$-collapsible.
\end{proof}

\begin{claim}\label{claim:link}
The complex $\lk(X,v)$ is $(d(X,\mathcal{A},<)-1)$-collapsible.
\end{claim}
\begin{proof}
 Let $I=\{i\in[m]:\, v\in \sigma_i\}$.
For every $i\in I$, let $\sigma_i''=\sigma_i\setminus \{v\}$. Write $I=\{i_1,\ldots,i_r\}$, where $i_1<\cdots<i_r$, and let $\mathcal{A}''=(\sigma_{i_1}'',\ldots,\sigma_{i_r}'')$.

For any $\sigma\in\lk(X,v)$, the simplex $\sigma\cup\{v\}$ belongs to $X$; hence,  there exists some $i\in[m]$ such that $\sigma\cup\{v\}\subset\sigma_i$. Since $v\in \sigma\cup\{v\}$, we must have $i\in I$, and therefore $\sigma\subset \sigma_i''=\sigma_i\setminus\{v\}$. So, by the induction hypothesis, $\lk(X,v)$ is $d(\lk(X,v),\mathcal{A}'',<)$-collapsible.

Let $\sigma\in \lk(X,v)$. We will show that \[
M_{X,\mathcal{A},<}(\sigma\cup\{v\})=M_{\lk(X,v),\mathcal{A}'',<}(\sigma)\cup\{v\}.\]
Let
\[
\mes_{X,\mathcal{A},<}(\sigma\cup\{v\})=(v_1,\ldots,v_n),\]
and
\[\mes_{\lk(X,v),\mathcal{A}'',<}(\sigma)=(u_1,\ldots,u_t).\]
For any $j\in[r]$,  $\sigma\subset \sigma_{i_j}''$ if and only if $\sigma\cup\{v\}\subset \sigma_{i_j}$. Also, for $i\notin I$, $\sigma\cup\{v\}\not\subset\sigma_i$ (since $v\notin\sigma_i$). Therefore, $n=i_{t+1}-1$. 

The vertex $v$ is the minimal vertex in $V\setminus \sigma_1$, therefore it is the minimal vertex in $(\sigma\cup\{v\})\setminus\sigma_1$. Hence, we have $v_1=v$. Now, let $i>1$ such that $i\notin I$. Then, $v_1=v$ is the vertex of minimal index in the sequence $(v_1,\ldots,v_{i-1})$ that is contained in $(\sigma\cup\{v\})\setminus \sigma_i$. Therefore, $v_i=v$.

Finally, we will show that $v_{i_j}=u_j$ for all $j\in[t]$. We argue by induction on $j$. Let $j\in[t]$, and assume that $v_{i_{\ell}}=u_{\ell}$ for all $\ell<j$.

For any $k<i_j$, either $v_k=v$ (if $k\notin I$) or $v_k=u_{\ell}$ for some $\ell<j$ (if $k=i_{\ell}\in I$).
Also, since $v\in \sigma_{i_j}$, we have $(\sigma\cup\{v\})\setminus\sigma_{i_j}= \sigma\setminus \sigma_{i_j}''$.
So, for any $k<i_j$, $v_k\in (\sigma\cup\{v\})\setminus\sigma_{i_j}$ if and only if $k=i_\ell$ for some $\ell<j$ such that $u_{\ell}\in \sigma\setminus\sigma_{i_j}''$.
Therefore, $v_{i_j}$ is old at $i_j$ if and only if $u_j$ is old at $j$, and if $v_{i_j}$ and $u_j$ are both old, then $v_{i_j}=u_j$. Otherwise, assume that $v_{i_j}$ is new at $i_j$ and $u_j$ is new at $j$. Then, $v_{i_j}$ is the minimal vertex in $(\sigma\cup\{v\})\setminus \sigma_{i_j}$, and $u_j$ is the minimal vertex in  $\sigma\setminus \sigma_{i_j}''=(\sigma\cup\{v\})\setminus\sigma_{i_j}$. Thus, $v_{i_j}=u_j$.

So, for any $\sigma\in \lk(X,v)$ we obtain  
\[
|M_{\lk(X,v),\mathcal{A}'',<}(\sigma)|=|M_{X,\mathcal{A},<}(\sigma\cup\{v\})|-1.\]
Hence, 
\[
d(\lk(X,v),\mathcal{A}'',<)\leq d(X,\mathcal{A},<)-1.
\]
So, $\lk(X,v)$ is $(d(X,\mathcal{A},<)-1)$-collapsible.

\end{proof}

By Claim \ref{claim:deletion}, Claim \ref{claim:link} and Lemma \ref{lemma:tancer}, $X$ is $d(X,\mathcal{A},<)$-collapsible.

\end{proof}

\begin{proof}[Proof of Theorem \ref{thm:mes_simpler}]
Let $<$ be some linear order on the vertex set $V$, and let $\mathcal{A}=(\sigma_1,\ldots,\sigma_m)$ be the sequence of maximal faces of $X$ (ordered in any way).

Let $i\in [m]$ and let $\sigma\in X$ with $m_{X,\mathcal{A},<}(\sigma)=i$.
Let $\mes_{X,\mathcal{A},<}(\sigma)=(v_1,\ldots,v_{i-1})$. Then $M_{X,\mathcal{A},<}(\sigma)=\{v_{i_1},\ldots,v_{i_k}\}$ for some $i_1<\cdots<i_k\in [i-1]$ (these are exactly the indices $i_j$ such that $v_{i_j}$ is new at $i_j$). For each $j\in[k]$ we have $v_{i_j}\notin \sigma_{i_j}$.  In addition, since $v_{i_j}$ is new at $i_j$,  we have $v_{i_{\ell}}\in \sigma_{i_{j}}$ for all $\ell<j$. Let $i_{k+1}=i$. Since $m_{X,\mathcal{A},<}(\sigma)=i=i_{k+1}$, we have $\sigma\subset \sigma_{i_{k+1}}$. In particular, $v_{i_{\ell}}\in \sigma_{i_{k+1}}$ for all $\ell<k+1$.

Therefore, $M_{X,\mathcal{A},<}(\sigma)\in S(X)$. Thus, $d(X,\mathcal{A},<)\leq d'(X)$, and by Theorem \ref{prop:collapsible_by_mes}, $X$ is $d'(X)$-collapsible.
\end{proof}

\section{Collapsibility of independence complexes}
\label{sec:graphs}
Let $G=(V,E)$ be a graph. The independence complex $I(G)$ is the simplicial complex on vertex set $V$ whose simplices are the independent sets in $G$. 

\begin{definition}\label{def:k}
Let $k(G)$ be the maximal size of a set $\{v_1,\ldots,v_k\}\subset V$ that satisfies:
\begin{itemize}
    \item $\{v_i,v_j\}\notin E$ for all $i\neq j\in[k]$,
    \item There exist $u_1,\ldots,u_k\in V$ such that 
    \begin{itemize}
        \item $\{v_i,u_i\}\in E$ for all $i\in[k]$,
        \item $\{v_i,u_j\}\notin E$ for all $1\leq i<j\leq k$.
    \end{itemize}
\end{itemize}
\end{definition}

\begin{proposition}\label{prop:kisdprime}   
   $k(G)=d'(I(G))$.
\end{proposition}
\begin{proof}
	Let	 $A=\{v_1,\ldots,v_{k}\}\in S(I(G))$. Then, there exist maximal faces $\sigma_1,\ldots,\sigma_{k+1}$ of $I(G)$ such that:
	\begin{itemize}
	    \item $v_i\notin\sigma_i$ for all $i\in[k]$,
	    \item $v_i\in \sigma_j$ for all $1\leq i<j \leq k+1$.
	\end{itemize}
	Let $i\in[k]$. Since $\sigma_i$ is a maximal independent set in $G$ and $v_i\notin \sigma_i$, there exists some $u_{i}\in \sigma_{i}$ such that $\{v_{i},u_{i}\}\in E$.
	
	Let $1\leq i< j\leq k$. Since $v_{i}$ and $u_j$ are both contained in the independent set $\sigma_j$, we have $\{v_{i},u_{j}\}\notin E$. Furthermore, since  $A\subset\sigma_{k+1}$, $A$ is an independent set in $G$. That is, $\{v_{i},v_{j}\}\notin E$ for all $i\neq j\in[k]$. 
	So, $A$ satisfies the conditions of Definition \ref{def:k}.
	Hence, $|A|\leq k(G)$; therefore, $d'(I(G))\leq k(G)$.
	
    Now, let $k=k(G)$, and let $v_1,\ldots,v_k,u_1,\ldots,u_k\in V$ such that
    \begin{itemize}
    \item $\{v_i,v_j\}\notin E$ for all $i\neq j\in[k]$,
    \item $\{v_i,u_i\}\in E$ for all $i\in[k]$,
    \item $\{v_i,u_j\}\notin E$ for all $1\leq i<j\leq k$.
	\end{itemize}
	Let $i\in[k]$, and let $V_i=\{v_j:\, 1\leq j<i\}$. Note that $V_i\cup \{u_i\}$ forms an independent set in $G$; therefore, it is a simplex in $I(G)$. Let $\sigma_i$ be a maximal face of $I(G)$ containing $V_i\cup \{u_i\}$. Since $\{v_i,u_i\}\in E$, we have $v_i\notin \sigma_i$.
	
	The set $\{v_1,\ldots,v_k\}$ is also an independent set in $G$. Therefore, there is a maximal face $\sigma_{k+1}\in I(G)$ that contains it. 
	
    	By the definition of $\sigma_1,\ldots,\sigma_{k+1}$, we have $v_i\in \sigma_j$ for $1\leq i<j\leq k+1$. Therefore, $\{v_1,\ldots,v_k\}\in S(I(G))$; so, $k(G)=k\leq d'(I(G))$.
    	
    Hence, $k(G)=d'(I(G))$, as wanted.
	
\end{proof}

As an immediate consequence of Proposition \ref{prop:kisdprime} and Theorem \ref{thm:mes_simpler}, we obtain:
\begin{proposition}
	\label{prop:clique}
	The complex $I(G)$ is $k(G)$-collapsible.
\end{proposition}

Note that vertices $v_1,\ldots,v_k,u_1,\ldots,u_k\in V$ satisfying the conditions in Definition \ref{def:k} must all be distinct. As a simple corollary, we obtain
\begin{corollary}\label{cor:n2}
	The independence complex of a graph $G=(V,E)$ on $n$ vertices is $\left\lfloor  \frac{n}{2} \right\rfloor$- collapsible.
\end{corollary}

\section{Complexes of hypergraphs}
\label{sec:3}

In this section we prove our main results, Theorems \ref{thm:cov} and \ref{thm:int}.

\begin{proof}[Proof of Theorem \ref{thm:cov}]
	
	Let $\mathcal{H}$ be a hypergraph of rank $r$ on vertex set $[n]$, and let 
	\[
	\{A_1,\ldots, A_k\}\in S(\text{Cov}_{\mathcal{H},p}).
	\]
	Then, there exist maximal faces $\mathcal{F}_1,\ldots,\mathcal{F}_{k+1}\in \text{Cov}_{\mathcal{H},p}$ such that 
		\begin{itemize}
			\item $A_i\notin \mathcal{F}_i$ for all $i\in[k]$,
			\item $A_i\in \mathcal{F}_j$ for all $1\leq i<j\leq k+1$.
		\end{itemize}
For any $i\in[k+1]$, there is some $C_i\subset [n]$ of size at most $p$ that covers $\mathcal{F}_i$. Since $\mathcal{F}_i$ is maximal, then, for any $A\in\mathcal{H}$, $A\in\mathcal{F}_i$ if and only if $A\cap C_i\neq \emptyset$.
Therefore, we obtain
		\begin{itemize}
			\item $A_i\cap C_i=\emptyset$ for all $i\in [k]$,
			\item $A_i\cap C_j\neq \emptyset$ for all $1\leq i<j\leq k+1$.
		\end{itemize}	    
Hence, the pair of families
	\[\{A_{1},\ldots	A_{k},\emptyset\}\]
	and
	\[
	\{C_{1},\ldots,C_{k},C_{k+1}\}\]
	satisfies the conditions of Lemma \ref{lem:frankl}; thus, $k+1\leq \binom{r+p}{r}$. Therefore,
	\[
		d'(\hcov_{\mathcal{H},p})\leq \binom{r+p}{r}-1,
	\]
	and by Theorem \ref{thm:mes_simpler}, $\hcov_{\mathcal{H},p}$ is $\left(\binom{r+p}{r}-1\right)$-collapsible.

\end{proof}

\begin{proof}[Proof of Theorem \ref{thm:int}]

	Let $\mathcal{H}$ be a hypergraph of rank $r$ and let $G$ be the graph on vertex set $\mathcal{H}$ whose edges are the pairs $\{A,B\}\subset\mathcal{H}$ such that $A\cap B=\emptyset$. Then $\intcomplex_{\mathcal{H}}= I(G)$.
	
	Let $k=k(G)$ and let $\{A_1,\ldots, A_k\}\subset \mathcal{H}$ that satisfies the conditions of Definition \ref{def:k}. That is,
	\begin{itemize}
	    \item $A_i\cap A_j\neq \emptyset$ for all $i\neq j\in [k]$,
	    \item There exist $B_1,\ldots, B_k\in \mathcal{H}$ such that
	    \begin{itemize}
	    \item $A_i\cap B_i=\emptyset$ for all $i\in[k]$,
	    \item $A_i\cap B_j\neq \emptyset$ for all $1\leq i<j\leq k$.
	    \end{itemize}
	\end{itemize}
    
    Then, the pair of families \[\{A_{1},\ldots,A_{k},B_{k},\ldots, B_{1}\}\] and \[\{B_{1},\ldots,B_{k},A_{k},\ldots, A_{1}\}\] satisfies the conditions of Lemma \ref{lem:frankl}; therefore, $2k\leq \binom{2r}{r}$. Thus, by Proposition \ref{prop:clique}, $\intcomplex_{\mathcal{H}}=I(G)$ is $\frac{1}{2}\binom{2r}{r}$-collapsible.
\end{proof}

\section{More complexes of hypergraphs}
	\label{sec:4}
	
Let $\mathcal{H}$ be a hypergraph. A set $C$ is a \emph{$t$-transversal} of $\mathcal{H}$ if $|A\cap C|\geq t$ for all $A\in\mathcal{H}$. Let $\tau_t(\mathcal{H})$ be the minimal size of a $t$-transversal of $\mathcal{H}$. The hypergraph $\mathcal{H}$ is \emph{pairwise $t$-intersecting} if $|A\cap B|\geq t$ for all $A,B\in\mathcal{H}$.
Let
\[
	\hcov_{\mathcal{H},p}^t= \{ \mathcal{F}\subset\mathcal{H}: \, \tau_t(\mathcal{F})\leq p\}
\]
and
\[
	\intcomplex_{\mathcal{H}}^t=\{ \mathcal{F}\subset\mathcal{H}:\, \mathcal{F} \text{ is pairwise $t$-intersecting}\}.
\]

The following generalization of Lemma \ref{lem:frankl} was proved by F{\"u}redi in \cite{furedi1984geometrical}.

\begin{lemma}[F{\"u}redi \cite{furedi1984geometrical}]
	\label{lem:furedi}
	Let $\{A_1,\ldots, A_k\}$ and $\{B_1,\ldots, B_k\}$ be families of sets such that:
	\begin{itemize}
	    \item $|A_i|\leq r$, $|B_i|\leq p$ for all $i\in[k]$,
	    \item $|A_i\cap B_i|\leq t$ for all $i\in [k]$,
	    \item $|A_i\cap B_j|>t$ for all $1\leq i<j\leq k$.
	  \end{itemize}
	Then
	\[
	k\leq \binom{r+p-2t}{r-t}.
	\]
\end{lemma}

We obtain the following:

\begin{theorem}
	Let $\mathcal{H}$ be a hypergraph of rank $r$ and let $t\leq\min\{r,p\}-1$. Then $\hcov^{t+1}_{\mathcal{H},p}$ is $\left(\binom{r+p-2t}{r-t}-1\right)$-collapsible.
\end{theorem}

\begin{theorem}
	Let $\mathcal{H}$ be a hypergraph of rank $r$ and let $t\leq r-1$. Then  $\intcomplex^{t+1}_{\mathcal{H}}$ is $\frac{1}{2}\binom{2(r-t)}{r-t}$-collapsible.
\end{theorem}

Note that by setting $t=0$ we recover Theorems \ref{thm:cov} and \ref{thm:int}.
The proofs are essentially the same as the proofs of Theorems \ref{thm:cov} and \ref{thm:int}, except for the use of Lemma \ref{lem:furedi} instead of Lemma \ref{lem:frankl}.
The extremal examples are also similar: Let
\[
    \mathcal{H}_1=\left\{  A\cup[t] : \, A\in\binom{[r+p-t]\setminus[t]}{r-t}\right\}
\]
and
\[
    \mathcal{H}_2=\left\{  A\cup[t] : \,  A\in\binom{[2r-t]\setminus[t]}{r-t}\right\}.
\]
The complex $\hcov^{t+1}_{\mathcal{H}_1}$ is the boundary of the $\left(\binom{r+p-2t}{r-t}-1\right)$-dimensional simplex, hence it is not $\left(\binom{r+p-2t}{r-t}-2\right)$-collapsible, and the complex $\intcomplex^{t+1}_{\mathcal{H}_2}$ is the boundary of the $\frac{1}{2}\binom{2(r-t)}{r-t}$-dimensional cross-polytope, hence it is not $\left( \frac{1}{2}\binom{2(r-t)}{r-t}-1\right)$-collapsible.

Restricting ourselves to special classes of hypergraphs we may obtain better bounds on the collapsibility of their associated  complexes. For example, we may look at $r$-partite $r$-uniform hypergraphs (that is, hypergraphs $\mathcal{H}$ on vertex set $V=V_1\cupdot V_2\cupdot \cdots \cupdot V_r$ such that $|A\cap V_i|=1$ for all $A\in\mathcal{H}$ and $i\in[r]$). In this case we have the following result:

\begin{theorem}
\label{thm:rpartiteint}
Let $\mathcal{H}$ be an $r$-partite $r$-uniform hypergraph. Then $\intcomplex_{\mathcal{H}}$ is $2^{r-1}$-collapsible.
\end{theorem}

The next example shows that the bound on the collapsibility of $\intcomplex_{\mathcal{H}}$ in Theorem \ref{thm:rpartiteint} is tight: Let $\mathcal{H}$ be the complete $r$-partite $r$-uniform hypergraph with all sides of size $2$. It has $2^r$ edges, and any edge $A\in\mathcal{H}$ intersects all the edges of $\mathcal{H}$ except its complement. Therefore the complex $\intcomplex_{\mathcal{H}}$ is the boundary of the $2^{r-1}$-dimensional cross-polytope, so it is homeomorphic to a $(2^{r-1}-1)$-dimensional sphere. Hence, by Proposition \ref{claim:hom}, $\intcomplex_{\mathcal{H}}$ is not $(2^{r-1}-1)$-collapsible.

For the proof we need the following Lemma, due to Lov{\'a}sz, Ne{\v{s}}et{\v{r}}il and Pultr.

\begin{lemma}[Lov{\'a}sz, Ne{\v{s}}et{\v{r}}il, Pultr {\cite[Prop. 5.3]{lovasz1980product}}]
\label{lem:lov}
	Let $\{A_1,\ldots, A_k\}$ and $\{B_1,\ldots, B_k\}$ be families of subsets of $V=V_1\cupdot V_2 \cupdot \cdots \cupdot V_r$ such that:
	\begin{itemize}
	\item 	$|A_i\cap V_j|= 1$, $|B_i\cap V_j|= 1$ for all $i\in[k]$ and $j\in[r]$,
	\item $A_i\cap B_i=\emptyset$ for all $i\in [k]$, \item $A_i\cap B_j\neq \emptyset$ for all $1\leq i<j\leq k$.
	\end{itemize}
	Then
	\[
	k\leq 2^{r}.
	\]
\end{lemma}

A common generalization of Lemma \ref{lem:frankl} and Lemma \ref{lem:lov} was proved by Alon in \cite{alon1985extremal}. 

The proof of Theorem \ref{thm:rpartiteint} is the same as the proof of Theorem \ref{thm:int}, except that we replace Lemma \ref{lem:frankl} by Lemma \ref{lem:lov}. A similar argument was also used by Aharoni and Berger (\cite[Theorem 5.1]{aharoni2009rainbow}) in order to prove a related result about rainbow matchings in $r$-partite $r$-uniform hypergraphs. 

\subsection*{Acknowledgements}
I thank Professor Roy Meshulam for his guidance and help. I thank the anonymous referee for some helpful suggestions.

\end{document}